\documentclass[article,12pt]{amsart}

\usepackage[T1]{fontenc}
\usepackage{inputenc}
\usepackage[francais,english]{babel}
\usepackage{amsmath}
\usepackage{latexsym}
\usepackage{marvosym}
\usepackage{amsfonts}
\usepackage{amsthm}
\usepackage{float}
\usepackage{color}
\usepackage{epsfig}
\usepackage{graphicx}

\usepackage{amssymb}

\newcommand{\dR}{\mathbb{R}}
\newcommand{\dE}{\mathbb{E}}

\newcommand{\cN}{\mathcal{N}}

\newcommand{\cF}{\mathcal{F}}

\newcommand{\veps}{\varepsilon}
\newcommand{\rI}{\mathrm{I}}
\newcommand{\wh}{\widehat}

\newcommand{\ind}{\mbox{1}\kern-.25em \mbox{I}}
\font\calcal=cmsy10 scaled\magstep1
\def\build#1_#2^#3{\mathrel{\mathop{\kern 0pt#1}\limits_{#2}^{#3}}}
\def\liml{\build{\longrightarrow}_{}^{{\mbox{\calcal L}}}}

\def\videbox{\mathbin{\vbox{\hrule\hbox{\vrule height1ex \kern.5em
\vrule height1ex}\hrule}}}

\numberwithin{equation}{section}
\theoremstyle{plain}
\newtheorem{thm}{Theorem}[section]
\newtheorem{rem}{Remark}[section]
\newtheorem{lem}{Lemma}[section]

\email{philippe.fraysse@math.u-bordeaux1.fr}

\begin{document}
\title[Estimation of the shift parameter in regression models]
{Estimation of the shift parameter in regression models with unknown distribution of the observations}
\author{Philippe Fraysse}
\dedicatory{\normalsize Universit\'e de Bordeaux}
\address{Universit\'e Bordeaux 1, Institut de Math\'ematiques de Bordeaux, UMR CNRS 5251,
	351 cours de la lib\'eration, 33405 Talence cedex, France.}

\begin{abstract}
This paper is devoted to the estimation of the shift parameter in a semiparametric regression model when the distribution of the observation times is unknown. Hence, we propose to use a stochastic algorithm which takes
into account the estimation of the distribution of the observation times. We establish the almost sure convergence
of our estimator and the asymptotic normality. The main result of the paper is that, with little assumptions on the regularity of the regression function, the asymptotic variance obtained is the same as when the distribution is known. In that sense, we improve the recent work of Bercu and Fraysse \cite{BF10}.
\end{abstract}

\maketitle


\section{INTRODUCTION}

\noindent We propose to study the problem of the estimation of the shift parameter $\theta$ in the semi parametric regression model defined, for all $n\geq0$, by
\begin{equation}
\label{Sempar}
Y_{n}=f(X_{n}-\theta)+\varepsilon_{n},
\end{equation}
where $(X_n)$ and $(\veps_n)$ are two independent sequences of independent and identically distributed random variables.
 Model \eqref{Sempar} belongs to the family of shape invariant models introduced by Lawton \textit{et al.} \cite{Lawton}. One can find studies of that kind of models in the papers of Dalalyan \textit{et al.} \cite{DGT06}, of Gamboa \textit{et al.} \cite{GamboaLoubes07} or Vimond \cite{Vimond10}, whereas Castillo and Loub\`es \cite{CastilloLoubes09} and Trigano \textit{et al.} \cite{Trigano11} are interesting of such a model when the parameter $\theta$ is random. Recent advances on the subject have also provided by Bigot and Charlier \cite{BC11} and Bigot and Gendre \cite{BG10}.\\ 
Contrary to all the papers quoted previously, we are dealing with random observation times $(X_{n})$ and we assume that their distribution is unknown. 
Our goal is the estimation of $\theta$ in that case. More precisely, we propose to generalize the work of Bercu and Fraysse \cite{BF10} when the distribution of $(X_{n})$ is assumed to be known.
We implement a stochastic algorithm in order to estimate the unknown parameter $\theta$
without any preliminary evaluation of the regression function $f$. 
When the distribution of $(X_{n})$ is known, Bercu and Fraysse propose to use the algorithm similar to that of Robbins-Monro \cite{RobbinsMonro51}, defined, for all $n\geq0$, by
\begin{equation}
\label{RMalgogknown}
\widetilde{\theta}_{n+1}=\widetilde{\theta}_{n}+\gamma_{n}T_{n+1}
\end{equation}
where $(\gamma_n)$ is a positive sequence of real numbers decreasing towards zero and
$(T_n)$ is a sequence of random variables such that 
$\dE[T_{n+1}|\cF_n]=\phi(\widehat{\theta}_{n})$
where $\mathcal{F}_{n}$ stands for the $\sigma $-algebra of the events occurring up to time $n$. References on algorithm \eqref{RMalgogknown} can be found in \cite{BF10}.
Nevertheless, the expression of $T_{n+1}$ depends on the distribution of $(X_{n})$. To overcome this problem, we propose to replace the algorithm given by \eqref{RMalgogknown} by the one defined, for all $n\geq0$, by
\begin{equation}
\label{RMalgo}
\widehat{\theta}_{n+1}=\widehat{\theta}_{n}+\gamma_{n}\wh{T}_{n+1},
\end{equation}
where $\wh{T}_{n+1}$ depends only on an estimator of the distribution of $(X_{n})$ which will be explicited in the sequel. In particular, we no longer have  $\dE[\wh{T}_{n+1}|\cF_n]=\phi(\widehat{\theta}_{n})$. Algorithms of the form \eqref{RMalgo} have been studied by
Pelletier \cite{Pelletier98}, \cite{Pelletier298} where the author establishes convergence results under the hypothesis that $\left(\wh{T}_{n+1}-T_{n+1}\right)^2=o_{\mathbb{P}}(\gamma_{n})$. Nevertheless, in our situation, such an hypothesis is not verified and we can not apply this kind of convergence results.\\
The paper is organized as follows. Section 2 is devoted to the explanation of the estimation procedure of $\theta$. We establish the almost sure convergence of $\widehat{\theta}_{n}$ 
as well as the asymptotic normality under some little assumptions on the regularity of $f$. In particular, we establish that the asymptotic variance is the same as the one obtained in the paper \cite{BF10}, that is to say the estimation of the distribution of $(X_{n})$ does not disturb the asymptotic behaviour of $\widehat{\theta}_{n}$.
The proofs of the results are given is Section 3. 


\section{ESTIMATION PROCEDURE AND MAIN RESULTS}


\noindent 
We focus our attention on the estimation of the shift parameter $\theta$ in the semiparametric regression model given
by \eqref{Sempar}. We assume that $(\veps_n)$ is a sequence of independent and identically distributed random variables
with zero mean and unknown positive variance $\sigma^2$. Moreover, we add the two several hypothesis similar to that of
\cite{BF10}.
\begin{displaymath}
\begin{array}{ll}
(\mathcal{H}_1) & \textrm{The observation times $(X_n)$ are independent and identically distributed with}\\
& \textrm{unknown probability density function $g$, positive on its support $[-1/2,1/2]$.} \\
& \textrm{In addition, $g$ is continuous, twice differentiable with  bounded derivatives. }\\
& \textrm{We denote by $C_{g}>0$ the minimum of $g$ on $[-1/2,1/2]$. } \\
(\mathcal{H}_2) & \textrm{The shape function $f$ is symmetric, bounded, periodic with period 1}. 
 \end{array}
\end{displaymath}
When the density $g$ is known, Bercu and Fraysse \cite{BF10} propose to use the algorithm defined, for all $n\geq0$, by
\begin{equation}
\label{RMA}
\widetilde{\theta}_{n+1}=\pi_{C}\Bigl(\widetilde{\theta}_{n}+\text{sign}(f_1)\gamma_{n+1}T_{n+1}\Bigr),
\end{equation}
where the initial value $\widetilde{\theta}_{0} \in C$ and the random variable $T_{n+1}$ is defined by
\begin{equation*}
\label{DefT}
T_{n+1}=\frac{\sin(2\pi(X_{n+1}-\widetilde{\theta}_{n}))}{g(X_{n+1})}Y_{n+1}.
\end{equation*}
We recall that $\pi_{C}$ is the projection on the compact set $C=[-1/4;1/4]$ defined, for all $x\in \dR$, by
\begin{eqnarray*}
   \pi_{C}(x) = \left \{ \begin{array}{lll}
   \,\, \,x & \ \text{ if } \ |x|\leq 1/4, \vspace{1ex} \\
    \,1/4 & \ \text{ if } \ x\geq 1/4,  \vspace{1ex} \\
    \!\!-1/4 & \ \text{ if } \ x\leq -1/4. \\
   \end{array} \nonumber \right.
\end{eqnarray*} 
Moreover, we denote by the first Fourier coefficient of $f$
$$ f_1=\int_{-1/2}^{1/2}\cos(2\pi x)f(x)\,dx ,$$
and we define the function $\phi$, for all $x\in{\mathbb{R}}$, by
\begin{equation}
\label{deffunctionphi}
\phi(x)=\sin\left(2\pi\left(\theta-x\right)\right)f_{1}.
\end{equation}
Finally, $(\gamma_n)$ is a decreasing sequence of positive real numbers satisfying
\begin{equation}
\label{hypgamma}
\sum_{n=1}^\infty\gamma_{n}=+\infty
\hspace{1cm}\text{and}\hspace{1cm}
\sum_{n=1}^\infty\gamma_{n}^2<+\infty.
\end{equation}
When the density $g$ is unknown, it is not possible to use algorithm \eqref{RMA}. The idea is to replace $g$ in the expression of \eqref{RMA} by an estimator of $g$. More precisely, we study the algorithm defined, for all $n\geq0$, by
\begin{equation}
\label{RMAg}
\wh{\theta}_{n+1}=\pi_{C}\Bigl(\wh{\theta}_{n}+\text{sign}(f_1)\gamma_{n+1}\wh{T}_{n+1}\Bigr),
\end{equation}
with
\begin{equation}
\label{DefTg}
\wh{T}_{n+1}=\frac{\sin(2\pi(X_{n+1}-\wh{\theta}_{n}))}{\wh{g}_{n}(X_{n+1})}Y_{n+1}.
\end{equation}
where $\wh{g}_{n}$ is the Parzen-Rosenblatt kernel estimator of $g$ (see \cite{Tsybakov04}, \cite{WolvertonWagner} and \cite{Yamato71} for references) defined, for all $x\in{[-1/2;1/2]}$ and for all $n\geq0$, by
\begin{equation}
\label{PRrecursif}
\wh{g}_{n}(x)=\frac{1}{n}\sum_{i=1}^{n}\frac{1}{h_{i}}K\left(\frac{X_{i}-x}{h_{i}}\right).
\end{equation}
and where the kernel $K$ is a symmetric function, positive, with compact support and with
$$
\int_{-\infty}^{+\infty}K(x)dx=1,\hspace{5mm}\int_{-\infty}^{+\infty}K^{2}(x)dx=\mu^{2}<+\infty,\hspace{5mm}\frac{1}{2}\int_{-\infty}^{+\infty}x^{2}K(x)dx=\nu^{2}<+\infty.
$$
All the results which follow are based on the following lemma.
\begin{lem}
\label{supgnx}
If $h_{n}=n^{-\alpha}$ with $0<\alpha<1$, then, for all $(1+\alpha)/2<\beta<1$,
\begin{equation*}
\underset{|x|\leq{1/2}}\sup\left\vert\wh{g}_{n}(x)-g(x)\right\vert=\mathcal{O}\left(n^{-2\alpha}+n^{\beta-1}\right)\hspace{7mm}\textnormal{a.s.}
\end{equation*}
\end{lem}
\begin{proof}
The proof is given in Section 3. 
\end{proof}
\noindent In the sequel, we choose $h_{n}=n^{-\alpha}$ with $0<\alpha<1$ and $\gamma_{n}=1/n$. Our first result concerns the almost sure convergence of the estimator $\wh{\theta}_{n}$.
\begin{thm}
\label{thmascvg}
Assume that $(\mathcal{H}_1)$ and $(\mathcal{H}_2)$ hold and that
$|\theta|<1/4$. Then, if $K$ is a Lipschitz function, for all $0<\alpha<1$, $\wh{\theta}_{n}$ converges almost surely to $\theta$.\\
In addition, the number of times that 
the random variable $\wh{\theta}_{n}+\text{sign}(f_1)\gamma_{n+1}\wh{T}_{n+1}$ goes outside of $C$ 
is almost surely finite.
\end{thm}
\begin{proof} The proof is given in Section 3. 
\end{proof}
\noindent 
Before establishing the asymptotic normality of $\wh{\theta}_{n}$, we need the following lemma on the mean square error of $\wh{\theta}_{n}$.
\begin{lem}
\label{meanthetan}
Let $m\geq0$ and $\varepsilon>0$ such that $C_{g}>\varepsilon$ and define 
\begin{equation}
\label{defTm}
T_{m}=\inf\{n\geq m~:~|\wh{\theta}_{n}-\theta|\geq\varepsilon~~~\textnormal{ or }\underset{x\in{[-1/2;1/2]}}\inf\wh{g}_{n}(x)\leq{C_{g}-\varepsilon}\}.
\end{equation}
Suppose that $4\pi|f_{1}|>1$, then, for $0<\alpha<1/2$,
\begin{equation*}
\mathbb{E}\left[\left(\wh{\theta}_{n}-\theta\right)^2\rI_{\{T_{m}=+\infty\}}\right]
=\mathcal{O}\left(n^{-2\alpha}+n^{\frac{\alpha-1}{2}}\right).
\end{equation*}
\end{lem}
\begin{proof} The proof is given in Section 3. 
\end{proof}
\noindent 
In order to establish the asymptotic normality of $\wh{\theta}_{n}$, it is necessary to
introduce a second auxiliary function $\varphi$ defined, for all $t\in \dR$, by
\begin{eqnarray}
\label{defvarphi}
 \varphi(t)&=&\dE\Bigl[\frac{\sin^2(2\pi(X-t)) }{g^2(X)}(f^2(X-\theta) + \sigma^2)\Bigr], \\
 &=& \int_{-1/2}^{1/2}\frac{\sin^2(2\pi(x-t))}{g(x)}(f^2(x-\theta)+ \sigma^2)\,dx. \nonumber
\end{eqnarray}
As soon as $4\pi |f_{1}|>1$, denote
\begin{equation}
\label{varasympt}
\xi^{2}(\theta)=\frac{\varphi(\theta)}{4\pi |f_{1}|-1}.
\end{equation}
Moreover, we need to add the following hypothesis on the regularity of $f$. 
\begin{displaymath}
\begin{array}{ll}
\hspace{-0.8cm}(\mathcal{H}_3) & \textrm{The shape function $f$ is twice differentiable with bounded derivatives}. 
 \end{array}
\end{displaymath}
\begin{thm}
\label{thmcltrm}
Assume that $(\mathcal{H}_1)$, $(\mathcal{H}_2)$ and $(\mathcal{H}_3)$ hold and that
$|\theta|<1/4$. Moreover, suppose that $(\veps_{n})$ has a finite moment of order $>2$
and that $4\pi |f_{1}|>1$. Then, if $K$ is a Lipschitz function, we have the asymptotic normality, for $1/4<\alpha<1/2$,
\begin{equation}
\label{cltrm}
\sqrt{n}(\wh{\theta}_{n}-\theta) \liml \cN(0, \xi^2(\theta)).
\end{equation}
\end{thm}

\begin{proof} The proof is given in Section 3. 
\end{proof}


\begin{rem}
We could expect that asymptotic variance obtained in Theorem \ref{thmcltrm} be larger than asymptotic variance obtained when the density $g$ is supposed to be known. Nevertheless, comparing the result obtained here and the one obtained in Theorem 2.2 of \cite{BF10}, we point out that the asymptotic variances are exactly the same. Hence, with a little assumption on the regularity of $f$, estimation of the density $g$ does not disturb the estimation of $\theta$. In that sense, the estimation procedure of $\theta$ by algorithm \eqref{RMAg} improves the one proposed in \cite{BF10}.
\end{rem}

\vspace{2ex}


\section{PROOFS}


\subsection{Proof of Lemma \ref{supgnx}.}
For all $x\in{[-1/2;1/2]}$, we have
\begin{eqnarray}
\nonumber
\wh{g}_{n}(x)-g(x)&=&\wh{g}_{n}(x)-\mathbb{E}\left[\wh{g}_{n}(x)\right]+\mathbb{E}\left[\wh{g}_{n}(x)\right]-g(x)\\
\label{decompogchap}
&=&\frac{M_{n}(x)}{n}+\frac{R_{n}(x)}{n}
\end{eqnarray}
where, for all $n\geq1$ and for all $x\in{[-1/2;1/2]}$,
\begin{equation}
\label{defMngchap}
M_{n}(x)=\sum_{k=1}^{n}K_{h_{k}}\left(X_{k}-x\right)-\mathbb{E}\left[K_{h_{k}}\left(X_{k}-x\right)\right],
\end{equation}
and
\begin{equation}
\label{defRngchap}
R_{n}(x)=\sum_{k=1}^{n}\mathbb{E}\left[K_{h_{k}}\left(X_{k}-x\right)\right]-g(x),
\end{equation}
with, for $x\in{[-1/2;1/2]}$,
$$
K_{h_{k}}(x)=\frac{1}{h_{k}}K\left(\frac{x}{h_{k}}\right).
$$
Firstly, $(M_{n}(x))$ is a square integrable martingale whose increasing process is given, for all $n\geq 1$ and $x\in{[-1/2;1/2]}$, by
\begin{eqnarray*}
\label{defcrochetMngchap}
\langle M(x)\rangle_{n}&=&\sum_{k=1}^{n}\mathbb{E}\left[K_{h_{k}}\left(X_{k}-x\right)^{2}\right]-\mathbb{E}\left[K_{h_{k}}\left(X_{k}-x\right)\right]^{2}\\
&\leq&\sum_{k=1}^{n}\mathbb{E}\left[K_{h_{k}}\left(X_{k}-x\right)^{2}\right].
\end{eqnarray*}
However, for all $1\leq{k}\leq{n}$,
$$
\mathbb{E}\left[K_{h_{k}}\left(X_{k}-x\right)^{2}\right]=\frac{1}{h_{k}}\int_{\mathbb{R}}K^{2}\left(y\right)g(x+h_{k}y)dy.
$$
Hence, as $g$ is bounded, we deduce, for all $x\in{[-1/2;1/2]}$, that
\begin{equation}
\label{majocrochetMn}
\langle M(x)\rangle_{n}\leq\left\vert\left\vert g\right\vert\right\vert_{\infty}\mu^{2}\sum_{k=1}^{n}\frac{1}{h_{k}}\leq\left\vert\left\vert g\right\vert\right\vert_{\infty}\mu^{2}\frac{n}{h_{n}}\hspace{1cm}\textnormal{a.s.}
\end{equation}
Moreover, denote, for all $n\geq1$ and for all $x\in{[-1/2;1/2]}$, $\Delta M_{n}(x)=M_{n}(x)-M_{n-1}(x)$. Then, we have
\begin{equation}
\label{majodiffMn}
\left\vert\Delta M_{n}(x)\right\vert\leq{\frac{2}{h_{n}}\left\vert\left\vert K\right\vert\right\vert_{\infty}}.
\end{equation}
In particular, with the choice $h_{n}=1/n^{\alpha}$, we infer from \eqref{majocrochetMn} and from \eqref{majodiffMn} that there exists two constants $a$ and $b$ such that
\begin{equation}
\label{majocrochetdiff}
\langle M(0)\rangle_{n}\leq{a n^{1+\alpha}}\hspace{7mm}\textnormal{and}\hspace{7mm}\left\vert\Delta M_{n}(0)\right\vert\leq{b n^{\alpha}}.
\end{equation}
Moreover, as the kernel $K$ is bounded and Lipschitz, for all $\delta\in{]0;1[}$, there exists a constant $C_{\delta}$ such that, for all $x,y\in{\mathbb{R}}$,
\begin{equation}
\label{Lipschitznoyau}
\left\vert K(x)-K(y)\right\vert\leq{C_{\delta}\left\vert x-y\right\vert^{\delta}}.
\end{equation}
Thus, for all $x,y\in{[-1/2;1/2]}$, one obtain that
$$
\left\vert \Delta M_{n}(x)-\Delta M_{n}(y)\right\vert\leq{2 C_{\delta}\left\vert x-y\right\vert^{\delta}n^{\alpha\left(1+\delta\right)}}.
$$
In addition, for all $x,y\in{[-1/2;1/2]}$,
\begin{eqnarray*}
\langle M(x)-M(y)\rangle_{n}&\leq&\sum_{k=1}^{n}\mathbb{E}\left[\left(K_{h_{k}}\left(X_{k}-x\right))-K^{2}_{h_{k}}\left(X_{k}-y\right))\right)^{2}\right],\\
&\leq&\sum_{k=1}^{n}k^{2\alpha}\int_{\mathbb{R}}\left(K\left(k^{\alpha}\left(u-x\right)\right)-K\left(k^{\alpha}\left(u-y\right)\right)\right)^{2}g(u)du.
\end{eqnarray*}
With the change of variables $t=k^{\alpha}\left(u-x\right)$, one then obtain that
\begin{equation}
\label{diffMnvariablechange}
\langle M(x)-M(y)\rangle_{n}\leq\left\vert\left\vert g\right\vert\right\vert_{\infty}\sum_{k=1}^{n}k^{\alpha}\int_{\mathbb{R}}\left(K\left(t\right)-K\left(t+k^{\alpha}(x-y)\right)\right)^{2}dt.
\end{equation}
Consequently, we deduce from \eqref{Lipschitznoyau} that, for all $1\leq{k}\leq{n}$,
$$
\int_{\mathbb{R}}\left(K\left(t\right)-K\left(t+k^{\alpha}(x-y)\right)\right)^{2}dt\leq{2 C_{2\delta}\left\vert x-y\right\vert^{2\delta}k^{2\alpha\delta}}.
$$
Hence, it follows from \eqref{diffMnvariablechange} that, for all $x,~y\in{[-1/2;1/2]}$,
$$
\langle M(x)-M(y)\rangle_{n}\leq{2 C_{2\delta}\left\vert x-y\right\vert^{2\delta}n^{1+\alpha+2\alpha\delta}}.
$$
As $\delta$ ca be choosen as small as we want, the four conditions of Theorem 6.4.34 page 220 of \cite{Duflo97} are satisfactory, that is to say the martingale $(M_{n}(x))$ checks, for all $(1+\alpha)/2<\beta<1$,
\begin{equation}
\label{supMnx}
\underset{|x|\leq{1/2}}\sup\left\vert M_{n}(x)\right\vert=o(n^{\beta})\hspace{7mm}\textnormal{a.s.}
\end{equation}
Finally, it keeps to control the term $R_{n}(x)$ defined by \eqref{defRngchap}. However,
for all $x\in{[-1/2;1/2]}$ and for all $1\leq{k}\leq{n}$,
\begin{eqnarray*}
\left\vert\mathbb{E}\left[K_{h_{k}}\left(X_{k}-x\right)\right]-g(x)\right\vert&\leq&\int_{\mathbb{R}}K(y)\left\vert g(x+h_{k}y)dy-g(x)\right\vert\\
&\leq&h_{k}^{2}\left\vert\left\vert g^{\prime\prime}\right\vert\right\vert_{\infty}\nu^{2}.
\end{eqnarray*}
Hence,
\begin{equation}
\label{supRnx}
\underset{|x|\leq{1/2}}\sup\frac{\left\vert R_{n}(x)\right\vert}{n}=\mathcal{O}\left(\frac{1}{n}\sum_{k=1}^{n}h_{k}^{2}\right)=\mathcal{O}\left(h_{n}^{2}\right)\hspace{1cm}\textnormal{a.s.}
\end{equation}
The conjunction of \eqref{decompogchap}, \eqref{supMnx} and \eqref{supRnx} leads to, for all $(1+\alpha)/2<\beta<1$,
\begin{equation}
\underset{|x|\leq{1/2}}\sup\left\vert\wh{g}_{n}(x)-g(x)\right\vert=\mathcal{O}\left(n^{-2\alpha}+n^{\beta-1}\right)\hspace{7mm}\textnormal{a.s.}
\end{equation}
which concludes the proof.
$\hfill 
\mathbin{\vbox{\hrule\hbox{\vrule height1.5ex \kern.6em
\vrule height1.5ex}\hrule}}$

\subsection{Proof of Theorem \ref{thmascvg}.}
Without loss of generality, we suppose that $f_{1}>0$. Denote by $\mathcal{F}_{n}$ the sigma-algebra $\mathcal{F}_{n}=\sigma\left(X_{0},Y_{0},\hdots,X_{n},Y_{n}\right)$. We calculate the two first conditional moments of $\wh{T}_{n+1}$. On the one hand, for all $n\geq0$,
\begin{equation}
\label{meanTn+1}
\mathbb{E}\left[\wh{T}_{n+1}|\mathcal{F}_{n}\right]=\mathbb{E}\left[T_{n+1}|\mathcal{F}_{n}\right]+\mathbb{E}\left[\left(\wh{T}_{n+1}-T_{n+1}\right)|\mathcal{F}_{n}\right].\\
\end{equation}
where
\begin{equation}
\label{ancienT}
T_{n+1}=\frac{\sin\left(2\pi(X_{n+1}-\wh{\theta}_{n})\right)}{g(X_{n+1})}Y_{n+1}.
\end{equation}
Thanks (5.2) of \cite{BF10}, we have
$$
\mathbb{E}\left[T_{n+1}|\mathcal{F}_{n}\right]=\phi\left(\wh{\theta}_{n}\right)\hspace{5mm}\textnormal{a.s.},
$$
where the function $\phi$ is defined by \eqref{deffunctionphi}.
Hence, we deduce from \eqref{meanTn+1} that
\begin{equation}
\label{meanfinalTn+1}
\mathbb{E}\left[\wh{T}_{n+1}|\mathcal{F}_{n}\right]=\phi\left(\wh{\theta}_{n}\right)+\mathbb{E}\left[\left(\wh{T}_{n+1}-T_{n+1}\right)|\mathcal{F}_{n}\right]\hspace{5mm}\textnormal{a.s.}
\end{equation}
On the other hand, for all $n\geq0$,
\begin{eqnarray}
\nonumber
\mathbb{E}\left[\wh{T}_{n+1}^{2}|\mathcal{F}_{n}\right]&=&\mathbb{E}\left[\left(\wh{T}_{n+1}-T_{n+1}+T_{n+1}\right)^{2}|\mathcal{F}_{n}\right]\\
\label{mean2Tn+1}
&\leq&2\mathbb{E}\left[T_{n+1}^{2}|\mathcal{F}_{n}\right]+2\mathbb{E}\left[\left(\wh{T}_{n+1}-T_{n+1}\right)^{2}|\mathcal{F}_{n}\right].
\end{eqnarray}
Thanks (5.4) of \cite{BF10}, there exists a constant $M>0$ such that
$$
\underset{n\geq0}\sup~\mathbb{E}\left[T_{n+1}^{2}|\mathcal{F}_{n}\right]\leq{M}\hspace{5mm}\textnormal{a.s.}
$$
Hence, it follows from \eqref{mean2Tn+1} that, for all $n\geq0$,
\begin{equation}
\label{mean2finalTn+1}
\mathbb{E}\left[\wh{T}_{n+1}^{2}|\mathcal{F}_{n}\right]\leq 2M+2\mathbb{E}\left[\left(\wh{T}_{n+1}-T_{n+1}\right)^{2}|\mathcal{F}_{n}\right].
\end{equation}
Moreover, for all $n\geq0$, denote $V_{n}=\left(\wh{\theta}_{n}-\theta\right)^{2}$. Using the fact that $\pi_{K}$ is Lipschitz with constant $1$, we have, for all $n\geq0$,
\begin{equation}
\label{EquationRS1}
\mathbb{E}\left[V_{n+1}|\mathcal{F}_{n}\right]\leq V_{n}+\gamma_{n+1}^{2}\mathbb{E}\left[\wh{T}_{n+1}^{2}|\mathcal{F}_{n}\right]+2\gamma_{n+1}\left(\wh{\theta}_{n}-\theta\right)\mathbb{E}\left[\wh{T}_{n+1}|\mathcal{F}_{n}\right].
\end{equation}
Hence, it follows from \eqref{meanfinalTn+1}, \eqref{mean2finalTn+1} and the previous inequality \eqref{EquationRS1}, that
\begin{equation}
\label{EquationRS2}
\mathbb{E}\left[V_{n+1}|\mathcal{F}_{n}\right]\leq V_{n}+2\gamma_{n+1}^{2}\left(M+P_{n}\right)+2\gamma_{n+1}\left(\wh{\theta}_{n}-\theta\right)\phi\left(\wh{\theta}_{n}\right)+2\gamma_{n+1}Q_{n},
\end{equation}
where
\begin{equation}
\label{defPn}
P_{n}=\mathbb{E}\left[\left(\wh{T}_{n+1}-T_{n+1}\right)^{2}|\mathcal{F}_{n}\right]
\end{equation}
and
\begin{equation}
\label{defQn}
Q_{n}=\mathbb{E}\left[\left\vert\wh{T}_{n+1}-T_{n+1}\right\vert|\mathcal{F}_{n}\right].
\end{equation}
However, for all $n\geq0$,
\begin{eqnarray*}
\wh{T}_{n+1}-T_{n+1}&=&\sin\left(2\pi(X_{n+1}-\wh{\theta}_{n})\right)Y_{n+1}\left(\frac{1}{\wh{g}_{n}(X_{n+1})}-\frac{1}{g(X_{n+1})}\right)\\
&=&\frac{\sin\left(2\pi(X_{n+1}-\wh{\theta}_{n})\right)Y_{n+1}}{g(X_{n+1})\wh{g}_{n}(X_{n+1})}\left(g(X_{n+1})-\wh{g}_{n}(X_{n+1})\right).
\end{eqnarray*}
Since $g$ does not vanish on its support, $f$ is bounded and $\varepsilon_{n+1}$ is independent of $\mathcal{F}_{n}$ with finite moment of order $2$, we immediately deduce the existence of $C_{1}>0$ and $C_{2}>0$ such that
\begin{eqnarray}
\nonumber
P_{n}&=&\mathbb{E}\left[\frac{\sin^{2}\left(2\pi(X_{n+1}-\wh{\theta}_{n})\right)Y_{n+1}^{2}}{g^{2}(X_{n+1})\wh{g}_{n}^{2}(X_{n+1})}\left(g(X_{n+1})-\wh{g}_{n}(X_{n+1})\right)^{2}\Big|\mathcal{F}_{n}\right]\\
\label{majorTchap}
&\leq&C_{1} \mathbb{E}\left[\frac{\left(g(X_{n+1})-\wh{g}_{n}(X_{n+1})\right)^{2}}{\wh{g}_{n}^{2}(X_{n+1})}\Big|\mathcal{F}_{n}\right],
\end{eqnarray}
and
\begin{eqnarray}
\nonumber
Q_{n}&\leq&\mathbb{E}\left[\frac{\left\vert\sin\left(2\pi(X_{n+1}-\wh{\theta}_{n})\right)Y_{n+1}\right\vert}{g(X_{n+1})\wh{g}_{n}(X_{n+1})}\left\vert g(X_{n+1})-\wh{g}_{n}(X_{n+1})\right\vert\Big|\mathcal{F}_{n}\right],\\
\label{majorT2chap}
&\leq&C_{2}\mathbb{E}\left[\frac{\left\vert g(X_{n+1})-\wh{g}_{n}(X_{n+1})\right\vert}{\wh{g}_{n}(X_{n+1})}\Big|\mathcal{F}_{n}\right].
\end{eqnarray}
In addition, on the one hand
\begin{equation}
\label{conditional1}
 \mathbb{E}\left[\frac{\left(g(X_{n+1})-\wh{g}_{n}(X_{n+1})\right)^{2}}{\wh{g}_{n}^{2}(X_{n+1})}\Big|\mathcal{F}_{n}\right]=\int_{-1/2}^{1/2}\left(g(x)-\wh{g}_{n}(x)\right)^{2}\frac{g(x)}{\wh{g}_{n}^{2}(x)}dx,
\end{equation}
and on the other hand,
\begin{equation}
\label{conditional2}
 \mathbb{E}\left[\frac{\left\vert g(X_{n+1})-\wh{g}_{n}(X_{n+1})\right\vert}{\wh{g}_{n}(X_{n+1})}\Big|\mathcal{F}_{n}\right]=\int_{-1/2}^{1/2}\left\vert g(x)-\wh{g}_{n}(x)\right\vert\frac{g(x)}{\wh{g}_{n}(x)}dx.
\end{equation}
Then, it follows from Lemma \ref{supgnx} and the two previous calculations \eqref{conditional1} and \eqref{conditional2}  that, for all $(1+\alpha)/2<\beta<1$,
\begin{equation}
\label{finconditional1}
 \mathbb{E}\left[\frac{\left(g(X_{n+1})-\wh{g}_{n}(X_{n+1})\right)^{2}}{\wh{g}_{n}^{2}(X_{n+1})}\Big|\mathcal{F}_{n}\right]=\mathcal{O}\left(\left(n^{-4\alpha}+n^{2(\beta-1)}\right)\int_{-1/2}^{1/2}\frac{g(x)}{\wh{g}_{n}^{2}(x)}dx\right)\hspace{1cm}\textnormal{a.s.},
\end{equation}
and
\begin{equation}
\label{finconditional2}
 \mathbb{E}\left[\frac{\left\vert g(X_{n+1})-\wh{g}_{n}(X_{n+1})\right\vert}{\wh{g}_{n}(X_{n+1})}\Big|\mathcal{F}_{n}\right]=\mathcal{O}\left(\left(n^{-2\alpha}+n^{\beta-1}\right)\int_{-1/2}^{1/2}\frac{g(x)}{\wh{g}_{n}(x)}dx\right)\hspace{1cm}\textnormal{a.s.}
\end{equation}
We immediately deduce from \eqref{majorTchap} and \eqref{majorT2chap} that
\begin{equation}
\label{majorTchapfin}
 P_{n}=\mathcal{O}\left(\left(n^{-4\alpha}+n^{2(\beta-1)}\right)\int_{-1/2}^{1/2}\frac{g(x)}{\wh{g}^{2}_{n}(x)}dx\right)\hspace{1cm}\textnormal{a.s.}
\end{equation}
and
\begin{equation}
\label{majorT2chapfin}
Q_{n}=\mathcal{O}\left(\left(n^{-2\alpha}+n^{\beta-1}\right)\int_{-1/2}^{1/2}\frac{g(x)}{\wh{g}_{n}(x)}dx\right)\hspace{1cm}\textnormal{a.s.}
\end{equation}
Moreover, for all $x\in{[-1/2;1/2]}$,
$$
\underset{n\rightarrow{+\infty}}\lim \wh{g}_{n}(x)=g(x)\hspace{1cm}\textnormal{a.s.}
$$
Since $g$ and $\wh{g}_{n}$ are defined on the compact set $[-1/2;1/2]$, it follows from \eqref{majorTchapfin} and \eqref{majorT2chapfin} that
\begin{equation}
\label{majorPnfin}
 P_{n}=\mathcal{O}\left(n^{-4\alpha}+n^{2(\beta-1)}\right)\hspace{1cm}\textnormal{a.s.}
\end{equation}
and
\begin{equation}
\label{majorQnfin}
 Q_{n}=\mathcal{O}\left(n^{-2\alpha}+n^{\beta-1}\right)\hspace{1cm}\textnormal{a.s.}
\end{equation}
Finally, with the choice of step $\gamma_{n}=1/n$, we infer from the two previous equations that, for all $0<\alpha<1$,
\begin{equation}
\label{cvPnfin}
\sum_{n=0}^{+\infty}\gamma_{n+1}^{2}P_{n}<+\infty\hspace{1cm}\textnormal{a.s.}
\end{equation}
and
\begin{equation}
\label{cvQnfin}
\sum_{n=0}^{+\infty}\gamma_{n+1}Q_{n}<+\infty\hspace{1cm}\textnormal{a.s.}
\end{equation}
To conclude, we deduce from \eqref{cvPnfin} and \eqref{cvQnfin} together with \eqref{EquationRS2} and the Robbins-Siegmund Theorem (see \cite{Duflo97} page 18), that the sequence $(V_{n})$ converges almost surely to a finite random variable and
$$
\sum_{n=1}^{+\infty}\gamma_{n+1}\left(\wh{\theta}_{n}-\theta\right)\phi\left(\wh{\theta}_{n}\right)<+\infty
\hspace{1cm}\textnormal{a.s.}
$$
Following exactly the same lines as proof of Theorem 2.1 of \cite{BF10} from equation (5.6), we deduce that $\left(\wh{\theta}_{n}\right)$ converges almost surely to $\theta$ and that the number of times that 
the random variable $\wh{\theta}_{n}+\gamma_{n+1}\wh{T}_{n+1}$ goes outside of $C$ 
is almost surely finite.
$\hfill 
\mathbin{\vbox{\hrule\hbox{\vrule height1.5ex \kern.6em
\vrule height1.5ex}\hrule}}$
\subsection{Proof of Lemma \ref{meanthetan}}
Denote by $(W_{n})$ the sequence defined, for all $n\geq0$, by
$$
W_{n}=\frac{V_{n}}{\gamma_{n}}.
$$
Then, one deduce from \eqref{EquationRS2} and from the choice of step $\gamma_{n}=1/n$ that
\begin{equation}
\label{Equationmeantheta1}
\mathbb{E}\left[W_{n+1}|\mathcal{F}_{n}\right]\leq W_{n}\left(1+\gamma_{n}\right)+2\gamma_{n+1}\left(M+P_{n}\right)+2\left(\wh{\theta}_{n}-\theta\right)\phi\left(\wh{\theta}_{n}\right)
+2Q_{n}.
\end{equation}
Moreover, since \eqref{majorPnfin}, there exists a constant $L>0$ such that
\begin{equation}
\label{Equationmeantheta2}
\mathbb{E}\left[W_{n+1}|\mathcal{F}_{n}\right]\leq W_{n}\left(1+\gamma_{n}\right)+\gamma_{n+1}L+2\left(\wh{\theta}_{n}-\theta\right)\phi\left(\wh{\theta}_{n}\right)
+2Q_{n}.
\end{equation}
In addition, we have for all $x\in \dR$, $\phi(x)=2\pi f_{1}(\theta-x)+f_{1}(\theta-x)v(x)$ where
\begin{equation*}
v(x)=\frac{\sin(2 \pi (\theta -x) )-2 \pi (\theta -x)}{(\theta -x)}.
\end{equation*}
By the continuity of the function $v$, one can find $0<\varepsilon <1/2$ such that, if $|x -\theta |< \varepsilon$, 
\begin{equation}
\label{majv}
\frac{q}{2f_1} < v(x)<0.
\end{equation}
Hence, it follows from \eqref{Equationmeantheta2}  that for all $n\geq 1$, 
\begin{equation}
\label{Equationmeantheta3}
\mathbb{E}[W_{n+1}|\mathcal{F}_{n}]\leq{W_{n}+2\gamma_{n}W_{n}(q-f_{1}v(\wh{\theta}_{n}))+\gamma_{n}L+2Q_{n}}
\end{equation}
with $2q=1-4\pi f_{1}$ which means that $q<0$. 
Then, it follows from \eqref{defTm} and \eqref{majv} that 
\begin{equation}
\label{majvv}
0<-f_1 v(\wh{\theta}_{n})\rI_{\{T_{m}>n\}}< -\Bigl(\frac{q}{2}\Bigr)\rI_{\{T_{m}>n\}}.
\end{equation}
Hence, we deduce from the conjunction of \eqref{Equationmeantheta3}
and \eqref{majvv} that,
\begin{eqnarray}
\dE[W_{n+1}\rI_{\{T_{m}>n\}}|\cF_{n}]&\leq & W_{n}\rI_{\{T_{m}>n\}} + 2 \gamma_n W_n \rI_{\{T_{m}>n\}} \Bigl( q -\frac{q}{2} \Bigr)
+\gamma_{n}L+2Q_{n}\rI_{\{T_{m}>n\}}, \nonumber \\
\label{IF}
&\leq & W_{n}\rI_{\{T_{m}>n\}}(1+q\gamma_{n})+\gamma_{n}L+2Q_{n}\rI_{\{T_{m}>n\}}.
\end{eqnarray}
Since $\{T_{m}>n+1\}\subset \{T_{m}>n\}$, we obtain by taking the expectation on both sides
of \eqref{IF} that for all $n\geq m$,
\begin{equation}
\label{II}
\dE[W_{n+1}\rI_{\{T_{m}>n+1\}}]\leq (1+q\gamma_{n})\dE[W_{n}\rI_{\{T_{m}>n\}}]+\gamma_{n}L+2\mathbb{E}\left[Q_{n}\rI_{\{T_{m}>n\}}\right].
\end{equation}
From now on, denote $\alpha_{n}=\dE[W_{n}\rI_{\{T_{m}>n\}}]$. We infer from \eqref{II} that for all $n\geq m$,
\begin{equation}
\label{alpha}
\alpha_{n+1}\leq \beta_{n}\alpha_{m}+L\beta_{n}\sum_{k=m}^{n}\frac{\gamma_{k}}{\beta_{k}}+2\beta_{n}\sum_{k=m}^{n}\frac{1}{\beta_{k}}\mathbb{E}\left[Q_{k}\rI_{\{T_{m}>k\}}\right]
\hspace{1cm}\text{where}\hspace{1cm}
\beta_{n}=\prod_{k=m}^{n}(1+q\gamma_{k}).
\end{equation}
As $\gamma_n=1/n$, it follows from straightforward calculations that $\beta_{n}=O(n^q)$ and 
$$\sum_{k=1}^{n}\frac{\gamma_{k}}{\beta_{k}}=O(n^{-q}).$$
Moreover, one have from \eqref{majorT2chap} and \eqref{conditional2} that there exists $C_{2}>0$ such that,
\begin{equation}
\label{Qind}
Q_{n}\rI_{\{T_{m}>n\}}\leq C_{2}\int_{-1/2}^{1/2}\left\vert g(x)-\wh{g}_{n}(x)\right\vert\frac{g(x)}{\wh{g}_{n}(x)}\rI_{\{T_{m}>n\}}dx.
\end{equation}
However, for all $x\in{[-1/2;1/2]}$, we have
$$
\frac{g(x)}{\wh{g}_{n}(x)}\rI_{\{T_{m}>n\}}\leq{\frac{g(x)}{C_{g}-\varepsilon}}.
$$
Then, taking expectation on both sides of \eqref{Qind}, it follows from the previous inequality that
\begin{eqnarray}
\nonumber
\mathbb{E}\left[Q_{n}\rI_{\{T_{m}>n\}}\right]&\leq& C_{2}\int_{-1/2}^{1/2}\mathbb{E}\left[\left\vert g(x)-\wh{g}_{n}(x)\right\vert\right]\frac{g(x)}{C_{g}-\varepsilon}dx,\\
\label{majoQind}
&\leq& \frac{C_{2}}{C_{g}-\varepsilon}~\underset{x\in{[-1/2;1/2]}}\sup ~\mathbb{E}\left[\left\vert g(x)-\wh{g}_{n}(x)\right\vert\right].
\end{eqnarray}
The quantity $\mathbb{E}\left[\left\vert g(x)-\wh{g}_{n}(x)\right\vert\right]$ corresponds to the mean error of the recursive Parzen-Rosenblatt estimator. Hence, it is well-known that for $0<\alpha<1/2$,
$$
\underset{x\in{[-1/2;1/2]}}\sup ~\mathbb{E}\left[\left\vert g(x)-\wh{g}_{n}(x)\right\vert\right]=\mathcal{O}\left(n^{-2\alpha}+n^{\frac{\alpha-1}{2}}\right).
$$
Then, one deduce from \eqref{majoQind} that, for $0<\alpha<1/2$,
\begin{equation}
\label{majoQindfin}
\mathbb{E}\left[Q_{n}\rI_{\{T_{m}>n\}}\right]=\mathcal{O}\left(n^{-2\alpha}+n^{\frac{\alpha-1}{2}}\right),
\end{equation}
which implies that
$$
\beta_{n}\sum_{k=m}^{n}\frac{1}{\beta_{k}}\mathbb{E}\left[Q_{k}\rI_{\{T_{m}>k\}}\right]=\mathcal{O}\left(n^{-2\alpha+1}+n^{\frac{\alpha+1}{2}}\right).
$$
Thus, \eqref{alpha} together the previous equation implies that
$$
\alpha_{n}=\mathcal{O}\left(n^{-2\alpha+1}+n^{\frac{\alpha+1}{2}}\right).
$$
Hence, for all $m\geq0$,
\begin{equation}
\label{majWnind}
\dE[W_{n}\rI_{\{T_{m}=+\infty\}}]=\mathcal{O}\left(n^{-2\alpha+1}+n^{\frac{\alpha+1}{2}}\right),
\end{equation}
that is to say, for $0<\alpha<1/2$,
$$
\dE\left[\left(\wh{\theta}_{n}-\theta\right)^2\rI_{\{T_{m}=+\infty\}}\right]=\mathcal{O}\left(n^{-2\alpha}
+n^{\frac{\alpha-1}{2}}\right).
$$
$\hfill 
\mathbin{\vbox{\hrule\hbox{\vrule height1.5ex \kern.6em
\vrule height1.5ex}\hrule}}$
\subsection{Proof of Theorem \ref{thmcltrm}}
Without loss of generality, we suppose that $f_{1}>0$. We have the decomposition, for all $n\geq0$,
\begin{equation}
\label{decomposition1}
\wh{\theta}_{n+1}=\wh{\theta}_{n}+\gamma_{n+1}\left(\wh{T}_{n+1}
-\phi\left(\wh{\theta}_{n}\right)\right)+\gamma_{n+1}\phi\left(\wh{\theta}_{n}\right)+d_{n+1},
\end{equation}
where
\begin{equation}
\label{defdn+1}
d_{n+1}=\pi_{K}\left(\wh{\theta}_{n}+\gamma_{n+1}\wh{T}_{n+1}\right)
-\left(\wh{\theta}_{n}+\gamma_{n+1}\wh{T}_{n+1}\right).
\end{equation}
Moreover, as $\phi$ is two times differentiable, there exists $0<\xi_{n}<1$ such that
\begin{equation}
\label{IAF}
\phi\left(\wh{\theta}_{n}\right)=\left(\wh{\theta}_{n}-\theta\right)\phi^{\prime}\left(\theta\right)+\frac{\left(\wh{\theta}_{n}-\theta\right)^{2}}{2}\phi^{\prime\prime}\left(\theta+\xi_{n}\left(\wh{\theta}_{n}-\theta\right)\right).
\end{equation}
Then, it follows from \eqref{decomposition1} and \eqref{IAF} that, for all $n\geq0$,
\begin{equation}
\label{decompo2}
\wh{\theta}_{n+1}-\theta=\alpha_{n}\left(\wh{\theta}_{n}-\theta\right)
+\gamma_{n+1}\left(\wh{T}_{n+1}-\phi\left(\wh{\theta}_{n}\right)\right)
+\gamma_{n+1}r_{n}+d_{n+1},
\end{equation}
where
\begin{equation}
\label{defalpha}
\alpha_{n}=1+\gamma_{n+1}\phi^{\prime}\left(\theta\right),
\end{equation}
and
\begin{equation}
\label{defrn}
r_{n}=\frac{\left(\wh{\theta}_{n}-\theta\right)^{2}}{2}\phi^{\prime\prime}\left(\theta+\xi_{n}\left(\wh{\theta}_{n}-\theta\right)\right).
\end{equation}
In addition, for all $n\geq0$,
\begin{equation}
\label{TchapT}
\wh{T}_{n+1}=T_{n+1}+A_{n+1}+B_{n+1}+C_{n+1},
\end{equation}
where $T_{n+1}$ is given by \eqref{ancienT} and
\begin{equation}
\label{defAn}
A_{n+1}=\sin\left(2\pi(X_{n+1}-\theta)\right)Y_{n+1}\left(\frac{g(X_{n+1})-\wh{g}_{n}(X_{n+1})}{g^{2}(X_{n+1})}\right),
\end{equation}
\begin{equation}
\label{defBn}
B_{n+1}=\left(\sin\left(2\pi(X_{n+1}-\wh{\theta}_{n})\right)-\sin\left(2\pi(X_{n+1}-\theta)\right)\right)Y_{n+1}\left(\frac{g(X_{n+1})-\wh{g}_{n}(X_{n+1})}{g^{2}(X_{n+1})}\right)
\end{equation}
and
\begin{equation}
\label{defCn}
C_{n+1}=\sin\left(2\pi(X_{n+1}-\wh{\theta}_{n})\right)Y_{n+1}\left(\frac{\left(g(X_{n+1})-\wh{g}_{n}(X_{n+1})\right)^2}{g^{2}(X_{n+1})\wh{g}_{n}(X_{n+1})}\right).
\end{equation}
Then, denoting $D_{n+1}=A_{n+1}+B_{n+1}+C_{n+1}$, it follows from \eqref{TchapT}, that for all $n\geq0$,
\begin{equation}
\label{TchapT2}
\wh{T}_{n+1}=T_{n+1}+\mathbb{E}\left[A_{n+1}|\mathcal{F}_{n}\right]+D_{n+1}-\mathbb{E}\left[D_{n+1}|\mathcal{F}_{n}\right]+\mathbb{E}\left[B_{n+1}|\mathcal{F}_{n}\right]+\mathbb{E}\left[C_{n+1}|\mathcal{F}_{n}\right].
\end{equation}
Finally, we deduce from \eqref{decompo2} that, for all $n\geq0$,
\begin{eqnarray}
\nonumber
\wh{\theta}_{n+1}-\theta &=&\left(\alpha_{n}+\gamma_{n+1}\frac{\mathbb{E}\left[B_{n+1}|\mathcal{F}_{n}\right]}{\wh{\theta}_{n}-\theta}+\gamma_{n+1}\frac{r_{n}}{\wh{\theta}_{n}-\theta}\right)\left(\wh{\theta}_{n}-\theta\right)
+\gamma_{n+1}\left(T_{n+1}+\mathbb{E}\left[A_{n+1}|\mathcal{F}_{n}\right]-\phi\left(\wh{\theta}_{n}\right)\right)\\
\nonumber
&+&\gamma_{n+1}\left(D_{n+1}-\mathbb{E}\left[D_{n+1}|\mathcal{F}_{n}\right]\right)
+\gamma_{n+1}\mathbb{E}\left[C_{n+1}|\mathcal{F}_{n}\right]
+d_{n+1}.
\end{eqnarray}
An immediate recurrence in the previous equality leads to, for all $n\geq1$,
\begin{eqnarray}
\nonumber
\sqrt{n}\left(\wh{\theta}_{n}-\theta\right)&=&n^{1/2}\beta_{n-1}\left(\wh{\theta}_{1}-\theta\right)
+n^{1/2}\beta_{n-1}S_{n-1}+n^{1/2}\beta_{n-1}R^{1}_{n-1}\\
\label{decompofinal}
&+&n^{1/2}\beta_{n-1}R^{2}_{n-1}+n^{1/2}\beta_{n-1}R^{3}_{n-1},
\end{eqnarray}
where
\begin{equation}
\label{defBeta}
\beta_{n-1}=\prod_{k=1}^{n-1}\left(\alpha_{k}+\gamma_{k+1}\frac{\mathbb{E}\left[B_{k+1}|\mathcal{F}_{k}\right]}{\wh{\theta}_{k}-\theta}+\gamma_{k+1}\frac{r_{k}}{\wh{\theta}_{k}-\theta}\right),
\end{equation}
\begin{equation}
\label{defSn}
S_{n-1}=\sum_{k=1}^{n-1}\frac{\gamma_{k+1}}{\beta_{k}}\left(T_{k+1}-\phi\left(\wh{\theta}_{k}\right)+\mathbb{E}\left[A_{k+1}|\mathcal{F}_{k}\right]\right),
\end{equation}
\begin{equation}
\label{defRn1}
R^{1}_{n-1}=\sum_{k=1}^{n-1}\frac{\gamma_{k+1}}{\beta_{k}}\left(D_{k+1}-\mathbb{E}\left[D_{k+1}|\mathcal{F}_{k}\right]\right),
\end{equation}
\begin{equation}
\label{defRn2}
R^{2}_{n-1}=\sum_{k=1}^{n-1}\frac{\gamma_{k+1}}{\beta_{k}}\mathbb{E}\left[C_{k+1}|\mathcal{F}_{k}\right],
\end{equation}
and
\begin{equation}
\label{defRn}
R^{3}_{n-1}=\sum_{k=1}^{n-1}\frac{1}{\beta_{k}}d_{k+1}.
\end{equation}
We begin by finding a simple equivalent of the sequence $\beta_{n-1}$ given by \eqref{defBeta}. Firstly, one have, for all $n\geq0$,
$$
\mathbb{E}\left[B_{n+1}|\mathcal{F}_{n}\right]=\int_{-1/2}^{1/2}\frac{\left(\sin\left(2\pi(x-\wh{\theta}_{n})\right)-\sin\left(2\pi(x-\theta)\right)\right)}{g(x)}f(x-\theta)\left(g(x)-\wh{g}_{n}(x)\right)dx.
$$
Hence, as $f$ is bounded and $g$ does not vanish on $[-1/2;1/2]$, we deduce that there exists a constant $C>0$ such that, for all $n\geq1$,
\begin{eqnarray*}
\left|\frac{\mathbb{E}\left[B_{n+1}|\mathcal{F}_{n}\right]}{\wh{\theta}_{n}-\theta}\right|
&\leq&{\int_{-1/2}^{1/2}\left\vert\frac{\left(\sin\left(2\pi(x-\wh{\theta}_{n})\right)-\sin\left(2\pi(x-\theta)\right)\right)}{\wh{\theta}_{n}-\theta}\frac{f(x-\theta)}{g(x)}\left(g(x)-\wh{g}_{n}(x)\right)\right\vert dx}\\
&\leq&{C\sup_{-1/2\leq x\leq 1/2}\left|g(x)-\wh{g}_{n}(x)\right|}.
\end{eqnarray*}
In particular, thanks to Lemma \ref{supgnx} and with $\gamma_{n}=1/n$, 
\begin{equation}
\label{cvBn}
\sum_{n=1}^{+\infty}\gamma_{n+1}\left|\frac{\mathbb{E}\left[B_{n+1}|\mathcal{F}_{n}\right]}{\wh{\theta}_{n}-\theta}\right|<+\infty\hspace{8mm}\textnormal{a.s.}
\end{equation}
Secondly, by definition of $\phi$ given by \eqref{deffunctionphi}, one have for all $x\in{\mathbb{R}}$,
$$
\phi^{\prime\prime}(x)=-4\pi^2 f_{1}\sin\left(2\pi\left(\theta-x\right)\right).
$$
Hence, one can find a constant $C>0$ such that for all $n\geq0$,
$$
\left\vert\frac{\wh{\theta}_{n}-\theta}{2}\phi^{\prime\prime}\left(\theta+\xi_{n}\left(\wh{\theta}_{n}-\theta\right)\right)\right\vert\leq{C \left(\wh{\theta}_{n}-\theta\right)^2}.
$$
Thus, for all $n\geq1$,
$$
\sum_{n=1}^{+\infty}\gamma_{n+1}\left\vert\frac{\wh{\theta}_{n}-\theta}{2}\phi^{\prime\prime}\left(\theta+\xi_{n}\left(\wh{\theta}_{n}-\theta\right)\right)\right\vert\leq{C \sum_{n=1}^{+\infty}\gamma_{n+1}\left(\wh{\theta}_{n}-\theta\right)^2}.
$$
Consequently, with $\gamma_{n}=1/n$, one deduce from Lemma \ref{meanthetan} that, if $0<\alpha<1/2$, on $\{T_{m}=+\infty\}$, the sequence
$$
\sum_{k=1}^{n}\gamma_{k+1}\left\vert\frac{\wh{\theta}_{k}-\theta}{2}\phi^{\prime\prime}\left(\theta+\xi_{k}\left(\wh{\theta}_{k}-\theta\right)\right)\right\vert
$$
converges a.s. Moreover, as $\cup_{m=1}^{+\infty}\{T_{m}=+\infty\}$ is a set of probability $1$, it follows that, for $0<\alpha<1/2$,
\begin{equation}
\label{cvrn}
\sum_{n=1}^{+\infty}\gamma_{n+1}\left\vert\frac{\wh{\theta}_{n}-\theta}{2}\phi^{\prime\prime}\left(\theta+\xi_{n}\left(\wh{\theta}_{n}-\theta\right)\right)\right\vert<+\infty~~~~~
\textnormal{a.s.}
\end{equation}
Finally, one infer from \eqref{defBeta} together with \eqref{cvBn} and \eqref{cvrn} that there exists a constant $c>0$ such that
$$
\beta_{n-1}\sim c\prod_{k=1}^{n-1}\alpha_{k},
$$
that is to say
\begin{equation}
\label{asympbeta}
\beta_{n-1}\sim c n^{q},
\end{equation}
where $q:=\phi^{\prime}\left(\theta\right)=-2\pi f_{1}<-1/2.$
Finally, for $0<\alpha<1/2$, we infer from \eqref{asympbeta} that \eqref{decompofinal} is equivalent to
\begin{equation}
\label{decompofinal2}
\sqrt{n}\left(\wh{\theta}_{n}-\theta\right)=n^{1/2+q}\left(\wh{\theta}_{1}-\theta\right)
+n^{1/2+q}S_{n-1}+n^{1/2+q}R^{1}_{n-1}+n^{1/2+q}R^{2}_{n-1}+n^{1/2+q}R^{3}_{n-1},
\end{equation}
where in each sum $S_{n-1}$, $R^{1}_{n-1}$, $R^{2}_{n-1}$ and $R^{3}_{n-1}$, $\frac{\gamma_{k+1}}{\beta{k}}$ is replaced by $k^{-1-q}$.\\
Now, we are going to analyze the asymptotic behaviour of each term of \eqref{decompofinal2}.
Firstly, as $q<-1/2$, we immediately have
\begin{equation}
\label{cvtheta1}
n^{1/2+q}\left(\wh{\theta}_{1}-\theta\right)=o(1)\hspace{6mm}\textnormal{a.s.}
\end{equation}
Secondly, the sequence $(R_{n-1}^{3})$ is almost surely finite since the number of times that 
the random variable $\wh{\theta}_{n}+\text{sign}(f_1)\gamma_{n+1}\wh{T}_{n+1}$ goes outside of $C$ 
is almost surely finite (Theorem \ref{thmascvg}). Hence, as $q<-1/2$,
\begin{equation}
\label{cvRn4}
n^{1/2+q}R^{3}_{n-1}=\mathcal{O}(n^{1/2+q})=o(1)\hspace{6mm}\textnormal{a.s.}
\end{equation}
Thirdly, the sequence $(R_{n}^{1})$ is a square integrable martingale whose increasing process is given, for all $n\geq1$, by
\begin{eqnarray}
\nonumber
\langle R^{1}\rangle_{n-1}&=&\sum_{k=1}^{n-1}\frac{1}{k^{2+2q}}\left(\mathbb{E}\left[D_{k+1}^{2}|\mathcal{F}_{k}\right]-\mathbb{E}\left[D_{k+1}|\mathcal{F}_{k}\right]^{2}\right)\\
\label{majcrochetRn1}
&\leq&\sum_{k=1}^{n-1}\frac{1}{k^{2+2q}}\mathbb{E}\left[D_{k+1}^{2}|\mathcal{F}_{k}\right]\hspace{6mm}\textnormal{a.s.}
\end{eqnarray}
In addition, as $D_{k+1}=A_{k+1}+B_{k+1}+C_{k+1}$, one obtain that, for all $1\leq{k}\leq{n-1}$,
$$
\mathbb{E}\left[D_{k+1}^{2}|\mathcal{F}_{k}\right]\leq{4\left(\mathbb{E}\left[A_{k+1}^{2}|\mathcal{F}_{k}\right]
+\mathbb{E}\left[B_{k+1}^{2}|\mathcal{F}_{k}\right]+\mathbb{E}\left[C_{k+1}^{2}|\mathcal{F}_{k}\right]\right)}
$$
Hence, since $f$ is bounded, $g$ does not vanish on its support and $(\varepsilon_{n})$ has a moment of order $2$, one immediately deduce from \eqref{defAn}, \eqref{defBn} and \eqref{defCn} that
$$
\mathbb{E}\left[D_{n+1}^{2}|\mathcal{F}_{n}\right]=\mathcal{O}\left(\sup_{-1/2\leq x\leq 1/2}\left(g(x)-\wh{g}_{n}(x)\right)^{2}\right)\hspace{6mm}\textnormal{a.s.}
$$
Then, it follows from Lemma \ref{supgnx} and \eqref{majcrochetRn1} that, for all $n\geq1$,
$$
\langle R^{1}\rangle_{n-1}=\mathcal{O}\left(\sum_{k=1}^{n-1}\frac{1}{k^{2+2q}}(k^{-4\alpha}+k^{2\beta-2})\right)
=\mathcal{O}\left(\sum_{k=1}^{n-1}\frac{1}{k^{2+2q+4\alpha}}+\frac{1}{k^{4+2q-2\beta}}\right)
\hspace{6mm}\textnormal{a.s.}
$$
If $2+2q+4\alpha>1$ and $4+2q-2\beta>1$, $\langle R^{1}\rangle_{n-1}$ converges a.s., and then $ R^{1}_{n-1}$ converges a.s. \\
If $2+2q+4\alpha=1$ and $4+2q-2\beta=1$, then
$
\langle R^{1}\rangle_{n-1}=\mathcal{O}\left(\log(n)\right).
$
We then deduce from the strong law of large numbers for martingales that $ R^{1}_{n-1}=o(\log(n))$ a.s.\\
If 
$2+2q+4\alpha<1$ and $4+2q-2\beta<1$, then
\begin{equation*}
\langle R^{1}\rangle_{n-1}=
\mathcal{O}\left(\frac{1}{n^{1+2q+4\alpha}}+\frac{1}{n^{3+2q-2\beta}}\right)\hspace{6mm}\textnormal{a.s.}
\end{equation*}
Then, we infer from the strong law of large numbers for martingales given by Theorem 1.3.15 of \cite{Duflo97} that
for any $\gamma>0$, 
$$
R^{1}_{n-1}=\mathcal{O}\left(\frac{\log(n)^{1/2+\gamma/2}}{n^{1/2+q+2\alpha}}+\frac{\log(n)^{1/2+\gamma/2}}{n^{3/2+q-\beta}}\right)\hspace{6mm}\textnormal{a.s.}
$$
In the three cases, as $q<-1/2$, one can conclude that
\begin{equation}
\label{cvRn1}
n^{1/2+q}R^{1}_{n-1}=o(1)\hspace{6mm}\textnormal{a.s.}
\end{equation}
From the same way, one deduce from \eqref{defCn} that
$$
\mathbb{E}\left[C_{n+1}^{2}|\mathcal{F}_{n}\right]=\mathcal{O}\left(\sup_{-1/2\leq x\leq 1/2}\left(g(x)-\wh{g}_{n}(x)\right)^{2}\right)\hspace{6mm}\textnormal{a.s.}
$$
Hence, it follows from Lemma \ref{supgnx} that 
$$
R^{2}_{n-1}=\mathcal{O}\left(\sum_{k=1}^{n-1}\frac{1}{k^{1+q}}\sup_{-1/2\leq x\leq 1/2}\left|g(x)-\wh{g}_{k}(x)\right|^{2}\right)=\mathcal{O}\left(\sum_{k=1}^{n-1}\frac{1}{k^{1+q+4\alpha}}+\frac{1}{k^{3+q-2\beta}}\right)\hspace{6mm}\textnormal{a.s.}
$$
Thus, if $1+q+4\alpha>1$ and $3+q-2\beta>1$, the sequence $(R^{2}_{n-1})$ converges a.s. whereas if $1+q+4\alpha=1$ and $3+q-2\beta=1$, one obtain that
$$
R^{2}_{n-1}=\mathcal{O}\left(\log(n)\right)\hspace{6mm}\textnormal{a.s.}
$$
In the case where $1+q+4\alpha<1$ and $3+q-2\beta<1$, one deduce that
$$
R^{2}_{n-1}=\mathcal{O}\left(\frac{1}{n^{q+4\alpha}}+\frac{1}{n^{q+2-2\beta}}\right)\hspace{6mm}\textnormal{a.s.}
$$
Finally, in the three cases, one obtain that, if $\alpha>1/8$ and $\beta<3/4$,
\begin{equation}
\label{cvRn2}
n^{1/2+q}R^{2}_{n-1}=o(1)\hspace{6mm}\textnormal{a.s.}
\end{equation}
The hypothesis $\beta<3/4$ implies to take $\alpha<1/2$. Hence, one obtain from \eqref{decompofinal2} together with \eqref{cvtheta1}, \eqref{cvRn4}, \eqref{cvRn1} and \eqref{cvRn2}, that if $1/8<\alpha<1/2$,
\begin{equation}
\label{decompofinal3}
\sqrt{n}\left(\wh{\theta}_{n}-\theta\right)=n^{1/2+q}S_{n-1}+o(1)\hspace{6mm}\textnormal{a.s.}
\end{equation}
where we recall that, for all $n\geq1$,
\begin{equation}
\label{newSn-1}
S_{n-1}=\sum_{k=1}^{n-1}\frac{1}{k^{1+q}}\left(T_{k+1}-\phi\left(\wh{\theta}_{k}\right)+\mathbb{E}\left[A_{k+1}|\mathcal{F}_{k}\right]\right).
\end{equation}
However, one deduce from \eqref{deffunctionphi} and \eqref{ancienT} the decomposition, for $n\geq1$,
\begin{equation}
\label{decompMG}
\sum_{k=1}^{n-1}\frac{1}{k^{1+q}}\left(T_{k+1}-\phi\left(\wh{\theta}_{k}\right)\right)=M_{n-1}^{1}+M_{n-1}^{2}
\end{equation}
where
\begin{equation}
\label{defMn1}
M_{n-1}^{1}=\sum_{k=1}^{n-1}\frac{1}{k^{1+q}}\sin\left(2\pi(\theta-\wh{\theta}_{k})\right)\left(\frac{\cos\left(2\pi(X_{k+1}-\theta)\right)}{g(X_{k+1})}Y_{k+1}-f_{1}\right)
\end{equation}
and
\begin{equation}
\label{defMn2}
M_{n-1}^{2}=\sum_{k=1}^{n-1}\frac{1}{k^{1+q}}\cos\left(2\pi(\theta-\wh{\theta}_{k})\right)\frac{\sin\left(2\pi(X_{k+1}-\theta)\right)}{g(X_{k+1})}Y_{k+1}.
\end{equation}
The sequence $(M_{n-1}^{1})$ is a square integrable martingale whose increasing process is given, for all $n\geq1$, by
\begin{equation}
\label{defcrochetM1n}
\langle M^{1}\rangle_{n-1}=\sum_{k=1}^{n-1}\frac{1}{k^{2+2q}}\sin\left(2\pi(\theta-\wh{\theta}_{k})\right)^{2}
\mathbb{E}\left[\left(\frac{\cos\left(2\pi(X_{k+1}-\theta)\right)}{g(X_{k+1})}Y_{k+1}-f_{1}\right)^{2}|\mathcal{F}_{k}\right].
\end{equation}
Moreover, since $f$ is bounded, $g$ does not vanish on its support and $(\varepsilon_{k})$ has a moment of order $2$, one immediately obtain from \eqref{defcrochetM1n} that
$$
\langle M^{1}\rangle_{n-1}=\mathcal{O}\left(\sum_{k=1}^{n-1}\frac{1}{k^{2+2q}}\sin\left(2\pi(\theta-\wh{\theta}_{k})\right)^{2}\right)\hspace{6mm}\textnormal{a.s.}
$$
In addition, since $(\wh{\theta}_{n})$ converges almost surely to $\theta$ and $\sum_{k=1}^{n-1}\frac{1}{k^{2+2q}}=\mathcal{O}\left(n^{-1-2q}\right)$, we finally deduce that
$$
\langle M^{1}\rangle_{n-1}=o(n^{-1-2q})\hspace{6mm}\textnormal{a.s.}
$$
In addition, since $\varepsilon_{n}$ admits a moment of order $>2$, the sequence $(M_{n-1}^{1})$ checks a Lyapunov condition. Consequently, one can conclude from the central limit theorem for martingales given e.g. by Corollary 2.1.10 of \cite{Duflo97} that 
\begin{equation}
\label{cvMn1}
n^{1/2+q}M_{n-1}^{1}=o_{\mathbb{P}}(1).
\end{equation}
Hence, it follows from \eqref{decompMG} and \eqref{cvMn1} that
\begin{equation}
\label{decompMG2}
n^{1/2+q}\sum_{k=1}^{n-1}\frac{1}{k^{1+q}}\left(T_{k+1}-\phi\left(\wh{\theta}_{k}\right)\right)=M_{n-1}^{2}+o_{\mathbb{P}}(1),
\end{equation}
where $M_{n-1}^{2}$ is given by \eqref{defMn2}. Furthermore, since $(\wh{\theta}_{n})$ converges almost surely to $\theta$, one immediately deduce that the asymptotic behaviour of $M_{n}^{2}$ is the same as the one of the sequence $M_{n-1}^{3}$, given for all $n\geq1$, by
\begin{equation}
\label{defMn3}
M_{n-1}^{3}=\sum_{k=1}^{n-1}\frac{1}{k^{1+q}}\frac{\sin\left(2\pi(X_{k+1}-\theta)\right)}{g(X_{k+1})}Y_{k+1}.
\end{equation}
Finally, one obtain from \eqref{newSn-1} and \eqref{decompMG2} that, for $n\geq1$,
\begin{equation}
\label{decompMG3}
n^{1/2+q}S_{n-1}=n^{1/2+q}M_{n-1}^{3}+n^{1/2+q}\sum_{k=1}^{n-1}\frac{1}{k^{1+q}}\mathbb{E}\left[A_{k+1}|\mathcal{F}_{k}\right]+o_{\mathbb{P}}(1),
\end{equation}
where $M_{n-1}^{3}$ is given by \eqref{defMn3}.
In addition, \eqref{defAn} and the symmetry of $f$ leads to, for all $n\geq1$,
\begin{eqnarray}
\nonumber
\mathbb{E}\left[A_{n+1}|\mathcal{F}_{n}\right]&=&\int_{-1/2}^{1/2}\sin\left(2\pi(x-\theta)\right)f(x-\theta)\left(\frac{g(x)-\wh{g}_{n}(x)}{g(x)}\right)dx\\
\label{equationA}
&=&-\int_{-1/2}^{1/2}a(x)\wh{g}_{n}(x)dx.
\end{eqnarray}
where for all $-1/2\leq x \leq 1/2$,
\begin{equation}
\label{deffunctiona}
a(x)=\frac{\sin\left(2\pi(x-\theta)\right)f(x-\theta)}{g(x)}.
\end{equation}
Hence, we deduce from \eqref{PRrecursif} and the change of variables $u=\frac{X_{i}-x}{h_{i}}$, that
\begin{eqnarray}
\nonumber
\int_{-1/2}^{1/2}a(x)\wh{g}_{n}(x)dx&=&\frac{1}{n}\sum_{i=1}^{n}\int_{-1/2}^{1/2}a(x)\frac{1}{h_{i}}K\left(\frac{X_{i}-x}{h_{i}}\right)dx\\
\label{sumgn}
&=&\frac{1}{n}\sum_{i=1}^{n}\int_{\frac{-1/2-X_{i}}{h_{i}}}^{\frac{1/2-X_{i}}{h_{i}}}a(X_{i}+h_{i}u)K(u)du.
\end{eqnarray}
Moreover, since $X_{n}\in{[-1/2;1/2]}$, one have
\begin{equation}
\label{limps}
\frac{-1/2-X_{n}}{h_{n}}\underset{n\rightarrow{+\infty}}\longrightarrow{-\infty}\hspace{4mm}\textnormal{a.s.}
\hspace{5mm}\textnormal{ and }\hspace{5mm}
\frac{1/2-X_{n}}{h_{n}}\underset{n\rightarrow{+\infty}}\longrightarrow{+\infty}\hspace{4mm}\textnormal{a.s.}
\end{equation}
From now, denote by $[-A;A]$ the support of $K$. Then, we have, for all $n\geq1$,
\begin{eqnarray*}
\int_{\frac{-1/2-X_{n}}{h_{n}}}^{\frac{1/2-X_{n}}{h_{n}}}a(X_{n}+h_{n}u)K(u)du&=&\int_{-A}^{A}a(X_{n}+h_{n}u)K(u)du\\
&+&\int_{\frac{-1/2-X_{n}}{h_{n}}}^{-A}a(X_{n}+h_{n}u)K(u)du+\int_{A}^{\frac{1/2-X_{n}}{h_{n}}}a(X_{n}+h_{n}u)K(u)du.
\end{eqnarray*}
Then, we deduce from the previous equality and \eqref{sumgn} and \eqref{limps} that
\begin{equation}
\label{sumgn2}
\int_{-1/2}^{1/2}a(x)\wh{g}_{n}(x)dx=\frac{1}{n}\sum_{i=1}^{n}\int_{-A}^{A}a(X_{i}+h_{i}u)K(u)du+\mathcal{O}\left(\frac{1}{n}\right)\hspace{6mm}\textnormal{a.s.}
\end{equation}
Moreover, as $f$ and $g$ are two times differentiable, one can write a Taylor expansion of the function $a$ given by
\eqref{deffunctiona}. More precisely, there exists $0<\xi_{i}<1$ such that a.s.
$$
a(X_{i}+h_{i}u)=a(X_{i})+h_{i}ua^{\prime}(X_{i})+\frac{h_{i}^{2}}{2}u^{2}a^{\prime\prime}(X_{i}+\xi_{i}h_{i}u).
$$
Consequently, since $K$ is a symmetric density and $f$ and $g$ have bounded derivates, one infer from the previous equality and \eqref{sumgn2} that, as $\alpha<1/2$,
\begin{eqnarray}
\nonumber
\int_{-1/2}^{1/2}a(x)\wh{g}_{n}(x)dx&=&\frac{1}{n}\sum_{i=1}^{n}a(X_{i})+\mathcal{O}\left(\frac{1}{n}\sum_{i=1}^{n}h_{i}^{2}\right)+\mathcal{O}\left(\frac{1}{n}\right)\\
\label{sumgn3}
&=&\frac{1}{n}\sum_{i=1}^{n}a(X_{i})+\mathcal{O}\left(\frac{1}{n^{2\alpha}}\right)\hspace{6mm}\textnormal{a.s.}
\end{eqnarray}
Finally, one deduce from \eqref{equationA} together with \eqref{deffunctiona} and \eqref{sumgn3} that, as $\alpha<1/2$,
\begin{equation}
\label{equationAfin}
\mathbb{E}\left[A_{n+1}|\mathcal{F}_{n}\right]=-\frac{1}{n}\sum_{i=1}^{n}\frac{\sin\left(2\pi(X_{i}-\theta)\right)f(X_{i}-\theta)}{g(X_{i})}+\mathcal{O}\left(\frac{1}{n^{2\alpha}}\right)\hspace{6mm}\textnormal{a.s.}
\end{equation}
Moreover, if $1+q+2\alpha<1$, then for $\alpha>1/4$,
$$
n^{1/2+q}\sum_{k=1}^{n-1}\frac{1}{k^{1+q+2\alpha}}=o(1).
$$
If $1+q+2\alpha\geq1$, it is obvious that
$$
n^{1/2+q}\sum_{k=1}^{n-1}\frac{1}{k^{1+q+2\alpha}}=o(1).
$$
Hence, it follows from \eqref{equationAfin} and the two previous equality that
\begin{equation}
\label{equationAfin2}
n^{1/2+q}\sum_{k=1}^{n-1}\frac{1}{k^{1+q}}\mathbb{E}\left[A_{k+1}|\mathcal{F}_{k}\right]=-n^{1/2+q}\sum_{k=1}^{n-1}\frac{1}{k^{2+q}}\sum_{i=1}^{k}\frac{\sin\left(2\pi(X_{i}-\theta)\right)f(X_{i}-\theta)}{g(X_{i})}+o(1)\hspace{6mm}\textnormal{a.s.}
\end{equation}
Finally, the conjunction of \eqref{decompofinal3} together with \eqref{decompMG3} and \eqref{equationAfin2} let us to conclude that, for $1/4<\alpha<1/2$,
\begin{eqnarray}
\nonumber
\sqrt{n}\left(\wh{\theta}_{n}-\theta\right)&=&n^{1/2+q}M_{n-1}^3+n^{1/2+q}\sum_{k=1}^{n-1}\frac{1}{k^{2+q}}\sum_{i=1}^{k}\frac{\sin\left(2\pi(X_{i}-\theta)\right)f(X_{i}-\theta)}{g(X_{i})}+o_{\mathbb{P}}(1)\\
\label{decompofinal4}
&=&n^{1/2+q}\left(\sum_{k=1}^{n-1}\frac{1}{k^{1+q}}v(X_{k+1})Y_{k+1}-\sum_{k=1}^{n-1}\frac{1}{k^{2+q}}\sum_{i=1}^{k}v(X_{i})f(X_{i}-\theta)\right)+o_{\mathbb{P}}(1).
\end{eqnarray}
where for all $-1/2\leq{x}\leq{1/2}$, 
\begin{equation}
\label{deffunctionv}
v(x)=\frac{\sin\left(2\pi(x-\theta)\right)}{g(x)}.
\end{equation}
From now, denote for all $1\leq{k}\leq{n-1}$,
\begin{equation}
\label{defvectorU}
U_{k}=\begin{pmatrix}
X_{k}\\
\varepsilon_{k}
\end{pmatrix}.
\end{equation}
In the following, we note $U$ a random variable independent of the sequence $(U_{n})$ and with the same law as $U_{n}$. We obtain from \eqref{decompofinal4} that, for $1/4<\alpha<1/2$,
\begin{equation}
\label{decompofinal5}
\sqrt{n}\left(\wh{\theta}_{n}-\theta\right)=n^{1/2+q}\left(\sum_{k=1}^{n-1}\frac{1}{k^{1+q}}\psi_{1}(U_{k+1})-\sum_{k=1}^{n-1}\frac{1}{k^{2+q}}\sum_{i=1}^{k}\psi_{2}(U_{i})\right)+o_{\mathbb{P}}(1),
\end{equation}
where the function $\psi_{1}$ and $\psi_{2}$ are such that, for all $x=(x_{1},x_{2})\in{\mathbb{R}^2}$,
\begin{equation}
\label{deffunctionspsi}
\psi_{1}(x)=v(x_{1})\left(f(x_{1}-\theta)+x_{2}\right)
\hspace{6mm}
\textnormal{and}
\hspace{6mm}
\psi_{2}(x)=v(x_{1})f(x_{1}-\theta).
\end{equation}
In addition,
\begin{equation}
\label{decompophi2}
\sum_{k=1}^{n-1}\frac{1}{k^{2+q}}\sum_{i=1}^{k}\psi_{2}\left(U_{i}\right)=
\sum_{i=1}^{n-1}\sum_{k=i}^{n-1}\frac{1}{k^{2+q}}\psi_{2}\left(U_{i}\right)=
\sum_{i=1}^{n-1}s_{i}^{n-1}\psi_{2}\left(U_{i}\right)
\end{equation}
where for $1\leq{i}\leq{n-1}$,
\begin{equation}
\label{defsi}
s_{i}^{n-1}=\sum_{k=i}^{n-1}\frac{1}{k^{2+q}}.
\end{equation}
Thus, it follows from \eqref{decompophi2} that for $n\geq3$,
\begin{eqnarray}
\nonumber
\sum_{k=1}^{n-1}\frac{1}{k^{1+q}}\psi_{1}(U_{k+1})-\sum_{k=1}^{n-1}\frac{1}{k^{2+q}}\sum_{i=1}^{k}\psi_{2}(U_{i})
&=&\sum_{k=1}^{n-2}\left(\frac{1}{k^{1+q}}\psi_{1}(U_{k+1})-s_{k+1}^{n-1}\psi_{2}(U_{k+1})\right)\\
\label{diffpsi}
&+&\frac{1}{(n-1)^{1+q}}\psi_{1}(U_{n})-s_{1}^{n-1}\psi_{2}(U_{1}).
\end{eqnarray}
Moreover, since $\psi_{1}(U_{n})$ and $\psi_{2}(U_{n})$ are integrable, we have
$$
n^{1/2+q}\left(\frac{1}{(n-1)^{1+q}}\psi_{1}(U_{n})-s_{1}^{n-1}\psi_{2}(U_{1})\right)=o_{\mathbb{P}}(1).
$$
One finally deduce from \eqref{decompofinal5} and \eqref{diffpsi} that for $1/4<\alpha<1/2$,
\begin{eqnarray}
\nonumber
\sqrt{n}\left(\wh{\theta}_{n}-\theta\right)&=&n^{1/2+q}\sum_{k=1}^{n-2}\left(\frac{1}{k^{1+q}}\psi_{1}(U_{k+1})-s_{k+1}^{n-1}\psi_{2}(U_{k+1})\right)+o_{\mathbb{P}}(1),\\
\label{decompofinal6}
&=&\left(1~~1\right)n^{1/2+q}\mathcal{M}_{n}
+o_{\mathbb{P}}(1),
\end{eqnarray}
where for $n\geq3$,
\begin{equation}
\label{defMcaln}
\mathcal{M}_{n}=\sum_{k=1}^{n-2}\begin{pmatrix}
a_{k}\psi_{1}\left(U_{k+1}\right)\\
b_{k}\psi_{2}\left(U_{k+1}\right),
\end{pmatrix}
\end{equation}
and for $1\leq{k}\leq{n-2}$,
\begin{equation}
\label{defanbn}
a_{k}=\frac{1}{k^{1+q}}~~~~~~\textnormal{and}~~~~~b_{k}=-s_{k+1}^{n-1}.
\end{equation}
Moreover, as $f$ is symmetric and $(\varepsilon_{n})$ is of mean $0$, it is not hard to see that, for $1\leq{k}\leq{n-2}$,
$$
\mathbb{E}\left[\begin{pmatrix}
a_{k}\psi_{1}\left(U_{k+1}\right)\\
b_{k}\psi_{2}\left(U_{k+1}\right)
\end{pmatrix}\Big|\mathcal{F}_{k}\right]=
\mathbb{E}\left[\begin{pmatrix}
a_{k}\psi_{1}\left(U\right)\\
b_{k}\psi_{2}\left(U\right)
\end{pmatrix}\right]=0.
$$
Consequently, the sequence $(\mathcal{M}_{n})$ is a vectorial martingale whose increasing process is the matrix defined for $n\geq3$, by
\begin{equation}
\label{defcrochetM}
\langle\mathcal{M}\rangle_{n}=\sum_{k=1}^{n-2}\begin{pmatrix}
a_{k}^2\sigma_{1}^2&a_{k}b_{k}\sigma_{1,2}\\
a_{k}b_{k}\sigma_{1,2}&b_{k}^{2}\sigma_{2}^2\\
\end{pmatrix}.
\end{equation}
where
$$
\sigma_{1}^2=\mathbb{E}\left[\psi_{1}(U)^2\right],~~~~~~~~~~~\sigma_{2}^2=\mathbb{E}\left[\psi_{2}(U)^2\right]
~~~~~~\textnormal{and}~~~~~\sigma_{1,2}=\mathbb{E}\left[\psi_{1}(U)\psi_{2}(U)\right].
$$
Moreover, for $q<-1/2$,
\begin{equation}
\label{limsuman}
n^{1+2q}\sum_{k=1}^{n-2}a_{k}^{2}=n^{1+2q}\sum_{k=1}^{n-2}\frac{1}{k^{2+2q}}\underset{n\rightarrow{+\infty}}\longrightarrow \int_{0}^{1}\frac{dx}{x^{2+2q}}=-\frac{1}{1+2q}.
\end{equation}
In addition, for $n\geq3$,
\begin{equation*}
\sum_{k=1}^{n-2}a_{k}b_{k}=-\sum_{k=1}^{n-2}\frac{1}{k^{1+q}}s_{k+1}^{n-1}
=-\sum_{k=1}^{n-2}\sum_{i=k+1}^{n-1}\frac{1}{k^{1+q}}\frac{1}{i^{2+q}}
=-\sum_{i=2}^{n-1}\sum_{k=1}^{i-1}\frac{1}{k^{1+q}}\frac{1}{i^{2+q}}.
\end{equation*}
Hence, one deduce from the Toeplitz lemma that, as $q<-1/2$,
\begin{equation}
\label{limsumanbn}
n^{1+2q}\sum_{k=1}^{n-2}a_{k}b_{k}=-\frac{1}{n}\sum_{i=2}^{n-1}\frac{n^{2+2q}}{i^{2+2q}}\frac{1}{i}\sum_{k=1}^{i-1}\frac{i^{1+q}}{k^{1+q}}\underset{n\rightarrow{+\infty}}\longrightarrow -\int_{0}^{1}\frac{dx}{x^{1+q}}\int_{0}^{1}\frac{dx}{x^{2+2q}}=-\frac{1}{1+2q}\frac{1}{q}.
\end{equation}
Finally,
\begin{equation*}
\sum_{k=1}^{n-2}b_{k}^2=\sum_{k=1}^{n-2}\left(s_{k+1}^{n-1}\right)^2
=\sum_{k=1}^{n-2}\left(\sum_{i=k+1}^{n-1}\frac{1}{i^{2+q}}\right)^2
=\sum_{k=1}^{n-2}\sum_{i=k+1}^{n-1}\sum_{j=k+1}^{n-1}\frac{1}{i^{2+q}}\frac{1}{j^{2+q}}.
\end{equation*}
Consequently,
\begin{equation*}
\underset{n\rightarrow{+\infty}}\lim n^{1/2+q}\sum_{k=1}^{n-2}b_{k}^2
=\underset{n\rightarrow{+\infty}}\lim n^{1/2+q}\left(\sum_{i=2}^{n-1}\frac{1}{i^{2+q}}\sum_{j=2}^{i}\frac{1}{j^{1+q}}
+\sum_{j=3}^{n-1}\frac{1}{j^{2+q}}\sum_{i=2}^{j-1}\frac{1}{i^{1+q}}\right).
\end{equation*}
Hence, it immediately follows from \eqref{limsumanbn} that
\begin{equation}
\label{limsumbn}
n^{1+2q}\sum_{k=1}^{n-2}b_{k}^{2}\underset{n\rightarrow{+\infty}}\longrightarrow2 \frac{1}{1+2q}\frac{1}{q}.
\end{equation}
Hence, we infer from \eqref{defcrochetM} together with \eqref{limsuman}, \eqref{limsumanbn} and \eqref{limsumbn} that
\begin{equation}
\label{limcrochetM}
n^{1+2q}\langle\mathcal{M}\rangle_{n}\underset{n\rightarrow{+\infty}}\longrightarrow-\frac{1}{q(1+2q)}\begin{pmatrix}
q\sigma_{1}^2&\sigma_{1,2}\\
\sigma_{1,2}&-2\sigma_{2}^2\\
\end{pmatrix}\hspace{6mm}\textnormal{a.s.}
\end{equation}
In order to apply the central limit theorem for vectorial martingales, it remains to check the Lindeberg condition, that is to say, for all $\varepsilon>0$, 
$$
n^{1+2q}\sum_{k=1}^{n-2}\mathbb{E}\left[\left|\left|V_{k+1}\right|\right|^{2}\mathrm{1}_{\left|\left|V_{k+1}\right|\right|\geq\varepsilon n^{-1/2-q}}\Big|\mathcal{F}_{k}\right]\overset{\mathbb{P}}{\underset{n\rightarrow{+\infty}}\longrightarrow} 0
$$
where $V_{k+1}=\begin{pmatrix}
a_{k}\psi_{1}\left(U_{k+1}\right)\\
b_{k}\psi_{2}\left(U_{k+1}\right)
\end{pmatrix}$.
However, for $\delta>0$, one have
\begin{eqnarray}
\nonumber
\sum_{k=1}^{n-2}\mathbb{E}\left[\left|\left|V_{k+1}\right|\right|^{2}\mathrm{1}_{\left|\left|V_{k+1}\right|\right|\geq\varepsilon n^{-1/2-q}}\Big|\mathcal{F}_{k}\right]&=&
\sum_{k=1}^{n-2}\mathbb{E}\left[\frac{\left|\left|V_{k+1}\right|\right|^{2+\delta}}{\left|\left|V_{k+1}\right|\right|}\mathrm{1}_{\left|\left|V_{k+1}\right|\right|\geq\varepsilon n^{-1/2-q}}\Big|\mathcal{F}_{k}\right]\\
\label{majoLindebergV}
&\leq&\frac{1}{\varepsilon}n^{1/2+q}\sum_{k=1}^{n-2}\mathbb{E}\left[\left|\left|V_{k+1}\right|\right|^{2+\delta}\Big|\mathcal{F}_{k}\right]
\end{eqnarray}
Moreover, for $1\leq{k}\leq{n-2}$,
\begin{eqnarray}
\nonumber
\mathbb{E}\left[\left|\left|V_{k+1}\right|\right|^{2+\delta}\Big|\mathcal{F}_{k}\right]&=&\mathbb{E}\left[\left(a_{k}^2\psi_{1}\left(U_{k+1}\right)^{2}
+b_{k}^{2}\psi_{2}\left(U_{k+1}\right)^{2}\right)^{1+\delta/2}\Big|\mathcal{F}_{k}\right]\\
\label{majoVk+1}
&\leq&{
2^{\delta/2}\left(a_{k}^{2+\delta}
\mathbb{E}\left[\psi_{1}\left(U\right)^{2+\delta}\right]+b_{k}^{2+\delta}
\mathbb{E}\left[\psi_{2}\left(U\right)^{2+\delta}\right]\right)}.
\end{eqnarray}
Hence, as $(\varepsilon_{n})$ has a moment of order $>2$, one immediately deduce from \eqref{majoLindebergV} and \eqref{majoVk+1} that there
exists $\kappa_{\varepsilon}>0$ such that
\begin{equation}
\label{lindeberg1}
n^{1+2q}\sum_{k=1}^{n-2}\mathbb{E}\left[\left|\left|V_{k+1}\right|\right|^{2}\mathrm{1}_{\left|\left|V_{k+1}\right|\right|\geq\varepsilon n^{-1/2-q}}\Big|\mathcal{F}_{k}\right]\leq{\kappa_{\varepsilon} n^{3/2+3q}\sum_{k=1}^{n-2}\left(a_{k}^{2+\delta}+b_{k}^{2+\delta}\right)}.
\end{equation}
However,
\begin{equation}
\label{lindeberg2}
n^{3/2+3q}\sum_{k=1}^{n-2}\left(a_{k}^{2+\delta}+b_{k}^{2+\delta}\right)=n^{3/2+3q}\left(\sum_{k=1}^{n-2}\frac{1}{k^{(1+q)(2+\delta)}}+\sum_{k=1}^{n-2}\left(\sum_{i=k+1}^{n-1}\frac{1}{i^{2+q}}\right)^{2+\delta}\right).
\end{equation}
If $q=-1$, then
\begin{equation}
\label{q=-1}
n^{3/2+3q}\left(\sum_{k=1}^{n-2}\frac{1}{k^{(1+q)(2+\delta)}}+\sum_{k=1}^{n-2}\left(\sum_{i=k+1}^{n-1}\frac{1}{i^{2+q}}\right)^{2+\delta}\right)=
\mathcal{O}\left(n^{3/2-3}(n+n\log(n)^{2+\delta})\right)=o(1).
\end{equation}
If $q<-1$ then,
\begin{equation}
\label{q<-1}
n^{3/2+3q}\left(\sum_{k=1}^{n-2}\frac{1}{k^{(1+q)(2+\delta)}}+\sum_{k=1}^{n-2}\left(\sum_{i=k+1}^{n-1}\frac{1}{i^{2+q}}\right)^{2+\delta}\right)=
\mathcal{O}\left(\frac{n^{3/2+3q}}{n^{1+2q+q\delta+\delta}}\right)=\mathcal{O}\left(n^{1/2+q-\delta(q+1)}\right)=o(1).
\end{equation}
as soon as $\delta<\frac{1/2+q}{1+q}$. Moreover, in this case, $\frac{1/2+q}{1+q}>0$.\\
If $-1<q<-1/2$ there exists a constant $c>0$ such that
$$
\sum_{i=k+1}^{n-1}\frac{1}{i^{2+q}}\leq{\frac{c}{k^{1+q}}},
$$
which implies that
$$
n^{3/2+3q}\left(\sum_{k=1}^{n-2}\frac{1}{k^{(1+q)(2+\delta)}}+\sum_{k=1}^{n-2}\left(\sum_{i=k+1}^{n-1}\frac{1}{i^{2+q}}\right)^{2+\delta}\right)=
\mathcal{O}\left(n^{3/2+3q}\sum_{k=1}^{n-2}\frac{1}{k^{(1+q)(2+\delta)}}\right).
$$
In the case where $(1+q)(2+\delta)<1$ then,
\begin{equation}
\label{qautre1}
n^{3/2+3q}\sum_{k=1}^{n-2}\frac{1}{k^{(1+q)(2+\delta)}}=\mathcal{O}\left(\frac{n^{3/2+3q}}{n^{1+2q+q\delta+\delta}}\right)=\mathcal{O}\left(n^{1/2+q-\delta(q+1)}\right)=o(1)
\end{equation}
as soon as $\delta>\frac{1/2+q}{1+q}$, which is right because $\frac{1/2+q}{1+q}<0$. In the case where $(1+q)(2+\delta)\geq 1$, we clearly have
\begin{equation}
\label{qautre2}
n^{3/2+3q}\sum_{k=1}^{n-2}\frac{1}{k^{(1+q)(2+\delta)}}=o(1).
\end{equation}
Finally, one deduce from \eqref{majoLindebergV}, \eqref{majoVk+1} \eqref{lindeberg1} and \eqref{lindeberg2} together with \eqref{q=-1}, \eqref{q<-1}, \eqref{qautre1} and \eqref{qautre2} that
\begin{equation}
\label{lindebergfin}
n^{1+2q}\sum_{k=1}^{n-2}\mathbb{E}\left[\left|\left|V_{k+1}\right|\right|^{2}\mathrm{1}_{\left|\left|V_{k+1}\right|\right|\geq\varepsilon n^{-1/2-q}}\Big|\mathcal{F}_{k}\right]\underset{n\rightarrow{+\infty}}\longrightarrow 0\hspace{6mm}\textnormal{a.s.}
\end{equation}
that seems that the Lindeberg condition is satisfied. One can conclude from \eqref{limcrochetM} and from \eqref{lindebergfin} and the central limit theorem for martingales given e.g. in  Corollary 2.1.10 page 46 of \cite{Duflo97} that
\begin{equation}
\label{cvmathcalMn}
n^{1/2+q}\mathcal{M}_{n}\liml\mathcal{N}\left(0,\Gamma\right),
\end{equation}
where 
\begin{equation}
\label{defGamma}
\Gamma=-\frac{1}{q(1+2q)}\begin{pmatrix}
q\sigma_{1}^2&\sigma_{1,2}\\
\sigma_{1,2}&-2\sigma_{2}^2\\
\end{pmatrix}.
\end{equation}
Moreover, as $(\varepsilon_{n})$ is of mean $0$ and variance $\sigma^2$ and independent of $(X_{n})$, straightforward but tedious calculations lead to
$$
\sigma_{1}^2=\mathbb{E}\left[\psi_{1}(U)^2\right]=\int_{-1/2}^{1/2}\frac{\sin^{2}\left(2\pi(x-\theta)\right)}{g(x)}\left(f^{2}(x-\theta)+\sigma^{2}\right)dx,
$$
$$
\sigma_{2}^2=\mathbb{E}\left[\psi_{2}(U)^2\right]=\int_{-1/2}^{1/2}\frac{\sin^{2}\left(2\pi(x-\theta)\right)}{g(x)}f^{2}(x-\theta)dx,
$$
and 
$$
\sigma_{1,2}=\mathbb{E}\left[\psi_{1}(U)\psi_{2}(U)\right]=\sigma_{2}^2.
$$
Finally, \eqref{decompofinal6} together with \eqref{cvmathcalMn} and the Slutsky theorem let us to conclude that, if $1/4<\alpha<1/2$,
$$
\sqrt{n}\left(\wh{\theta}_{n}-\theta\right)\liml\mathcal{N}\left(0,-\frac{1}{1+2q}\sigma_{1}^2\right),
$$
which ends the proof of Theorem \ref{thmcltrm}.
$\hfill 
\mathbin{\vbox{\hrule\hbox{\vrule height1.5ex \kern.6em
\vrule height1.5ex}\hrule}}$
\nocite{*}
\bibliographystyle{acm}
\bibliography{g_unknown_new}

\begin{thebibliography}{10}

\bibitem{BF10}
{\sc Bercu, B., and Fraysse, P.}
\newblock A robbins-monro procedure for estimation in semiparametric regression
  models.
\newblock {\em Ann. Statist. 40\/} (2012).

\bibitem{BC11}
{\sc Bigot, J., and Charlier, B.}
\newblock On the consistency of {F}r\'echet means in deformable models for
  curve and image analysis.
\newblock {\em Electron. J. Stat. 5\/} (2011), 1054--1089.

\bibitem{BG10}
{\sc Bigot, J., and Gadat, S.}
\newblock A deconvolution approach to estimation of a common shape in a shifted
  curves model.
\newblock {\em Ann. Statist. 38}, 4 (2010), 2422--2464.

\bibitem{CastilloLoubes09}
{\sc Castillo, I., and Loubes, J.-M.}
\newblock Estimation of the distribution of random shifts deformation.
\newblock {\em Math. Methods Statist. 18 1\/} (2009), 21--42.

\bibitem{DGT06}
{\sc Dalalyan, A.~S., Golubev, G.~K., and Tsybakov, A.~B.}
\newblock Penalized maximum likelihood and semiparametric second-order
  efficiency.
\newblock {\em Ann. Statist. 34 1\/} (2006), 169--201.

\bibitem{Duflo97}
{\sc Duflo, M.}
\newblock {\em Random iterative models}, vol.~34 of {\em Applications of
  Mathematics}.
\newblock Springer-Verlag, Berlin, 1997.

\bibitem{GamboaLoubes07}
{\sc Gamboa, F., Loubes, J.-M., and Maza, E.}
\newblock Semi-parametric estimation of shifts.
\newblock {\em Electron. J. Stat. 1\/} (2007), 616--640.

\bibitem{HallHeyde80}
{\sc Hall, P., and Heyde, C.~C.}
\newblock {\em Martingale limit theory and its application}.
\newblock Academic Press Inc. New York, 1980.

\bibitem{Lawton}
{\sc Lawton, W.~H., Sylvestre, E.~A., and Maggio, M.~S.}
\newblock Self modeling nonlinear regression.
\newblock {\em Technometrics 14\/} (1972), 513--532.

\bibitem{Pelletier98}
{\sc Pelletier, M.}
\newblock On the almost sure asymptotic behaviour of stochastic algorithms.
\newblock {\em Stochastic Process. Appl. 78}, 2 (1998), 217--244.

\bibitem{Pelletier298}
{\sc Pelletier, M.}
\newblock Weak convergence rates for stochastic approximation with application
  to multiple targets and simulated annealing.
\newblock {\em Annals of Appli. Proba. 8}, 1 (1998), 10--44.

\bibitem{RobbinsMonro51}
{\sc Robbins, H., and Monro, S.}
\newblock A stochastic approximation method.
\newblock {\em Ann. Math. Statistics 22\/} (1951), 400--407.

\bibitem{Trigano11}
{\sc Trigano, T., Isserles, U., and Ritov, Y.}
\newblock Semiparametric curve alignment and shift density estimation for
  biological data.
\newblock {\em IEEE Trans. Signal Processing 59\/} (2011), 1970--1984.

\bibitem{Tsybakov04}
{\sc Tsybakov, A.~B.}
\newblock {\em Introduction \`a l'estimation non-param\'etrique}, vol.~41 of
  {\em Math\'ematiques \& Applications (Berlin)}.
\newblock Springer-Verlag, Berlin, 2004.

\bibitem{Vimond10}
{\sc Vimond, M.}
\newblock Efficient estimation for a subclass of shape invariant models.
\newblock {\em Ann. Statist. 38}, 3 (2010), 1885--1912.

\bibitem{WolvertonWagner}
{\sc Wolverton, C.~T., and Wagner, T.~J.}
\newblock Asymptotically optimal discriminant functions for pattern
  classification.
\newblock {\em IEEE Trans. Information Theory IT-15\/} (1969), 258--265.

\bibitem{Yamato71}
{\sc Yamato, H.}
\newblock Sequential estimation of a continuous probability density function
  and mode.
\newblock {\em Bull. Math. Statist. 14\/} (1970/71), 1--12; correction, ibid.
  15\ (1972), 133.

\end{thebibliography}

\end{document}